\documentclass[12pt]{amsart}
\usepackage{amsmath}
\usepackage{graphicx}
\usepackage{amsfonts}
\usepackage{amstext}
\usepackage{amsbsy}
\usepackage{amsopn}
\usepackage{amsxtra}
\usepackage{upref}
\usepackage{amsthm}
\usepackage{amsmath}
\usepackage{amssymb}

\newtheorem{thm}{Theorem}[section]
\newtheorem{cor}[thm]{Corollary}
\newtheorem{lemma}[thm]{Lemma}

\theoremstyle{definition}
\newtheorem{defn}[thm]{Definition}
\theoremstyle{remark}
\newtheorem{rem}[thm]{Remark}

\numberwithin{equation}{section}
\newtheorem*{thma}{Theorem A}
\newtheorem*{thmb}{Theorem B}
\newcommand{\Hess}{\operatorname{Hess}}
\newcommand{\grad}{\operatorname{grad}}
\newcommand{\wf}{\widetilde{f}}

\newcommand{\noi}{\noindent}
\newcommand{\sm}{\smallskip}
\newcommand{\me}{\medskip}
\newcommand{\bi}{\bigskip}
\newcommand{\RR}{\text{Re}}

\newcommand{\ra}{\rightarrow}

\def\D{{\mathbb D}}
\def\R{{\mathbb R}}
\def\S{\mathcal{S}}

\def\E{\mathcal{E}_{\wf}}
\newcommand{\sg}{\sigma}
\newcommand{\dl}{\delta}

\begin{document}

\title[Injectivity of minimal immersions]{Injectivity of minimal immersions and homeomorphic extensions to space}
\author{Martin Chuaqui}
\thanks{The
author was partially supported by Fondecyt Grant  \#1150115.
\endgraf  {\sl Key words:} .
\endgraf {\sl 2000 AMS Subject Classification}. Primary: 53A10, 53A30;\,
Secondary: 30C35.}
%
%\address{Facultad de Matem\'aticas\\ Pontificia Universidad Cat\'olica de Chile\\
%Casilla 306, Santiago 22, CHILE.} \email{mchuaqui@mat.puc.cl}

%\thanks{}%
%\subjclass{}%
%\keywords{}%

%\date{}%
%\dedicatory{}%
%\commby{}%
% ----------------------------------------------------------------
\begin{abstract}
We study a recent general criterion for the injectivity of the conformal immersion of a
Riemannian manifold into higher dimensional Euclidean space, and show how it gives rise to important conditions
for Weierstrass-Ennerper lifts defined in the unit disk $\mathbb{D}$ endowed with a conformal metric.
Among the corollaries, we obtain a Becker type
condition and a sharp condition depending on the Gaussian curvature and the diameter for an immersed
geodesically convex minimal disk in $\mathbb{R}^3$ to be embedded. Extremal configurations for the criteria are
also determined, and can only occur on a catenoid. For  non-extremal configurations, we establish
fibrations of space by circles in domain and range that give a geometric analogue of the Ahlfors-Weill extension.

\end{abstract}
\maketitle

\section{Introduction}

In recent years, several criteria have been derived for the injectivity of conformal immersions of planar domains
into higher dimensional Euclidean spaces. Important particular cases consider Weierstrass-Enneper lifts of harmonic mappings
and holomorphic immersions into $\mathbb{C}^n$ \cite{cdo:harmonic lift}, \cite{cdo:holomorphic}. The criteria represent
extensions of the classical Nehari theory for homomorphic mappings in one complex variable \cite{Ne}, and are made possible
through appropriate generalizations of the notion of a Schwarzian derivative. It is interesting to observe that the new criteria
do not depend alone on the size of the generalized (conformal) Schwarzian, because the second fundamental form of the immersed surface must also
be taken into account. A key ingredient in this development has been Ahlfors' definition of a Schwarzian derivative for
parametrized curves in Euclidean spaces, in particular, in connection with the injectivity criterion found in
\cite{chuaqui-gevirtz}. This one-dimensional operator  brings in  both the conformal Schwarzian as well
as the the second fundamental form. In \cite{stowe:a-d}, the author introduces the  more general  {\it Ahlfors derivative} for conformal immersions,
which combines Ahlfors'
Schwarzian for curves and the conformal Schwarzian.  Corollary 12,
on which we will concentrate,  represents
one of the most general formulations of a criterion for the inyectivity of a conformal immersion
of a Riemannian manifold into Euclidean space, and
we refer the reader to the paper for other interesting issues.

Our interest is the study of Corollary 12 when the Riemannian manifold takes the form of the unit disk
$\D$ endowed with a conformal metric, and when the immersion is a Weierstrass-Enneper lift. Suitable  choices of
conformal metrics render, among other, generalizations of conditions by Ahlfors \cite{AH},  by Becker \cite{Be},
and by Epstein \cite{Ep1}. Moreover, it gives way to a sharp condition depending just on the Gaussian curvature and the diameter for an immersed
geodesically convex minimal disk to be embedded.

As in the classical case, two additional elements appear of interest after injectivity has been established, namely, boundary behavior
of the lift and possible homeomorphic or quasiconformal extensions to space. These issues have been addressed before by considering
a  real-valued function associated in a canonical way to the lift  that measures up the conformal factor of the immersion
with that of the metric \cite{AW}. We will
offer a proof of Corollary 12 in the context described by appealing entirely to Ahlfors' Schwarzian for curves,
showing, in passing, a crucial convexity property of the canonical function. A continuous extension to the closed disk together with the analysis 
when injectivity can be lost at the boundary
will follow. We will apply ideas
developed in \cite{AW} to define a homeomorphic extension of the lift to the entire space, as a spatial
analogue of the Ahlfors-Weill construction.

\sm
The paper is organized as follows. In the remainder of the Introduction we give a brief account of the
main facts about harmonic mappings and Weierstrass-Enneper lifts. In Section 2, we lay out the background material
on the conformal Schwarzian that applies both for the lift and for a conformal metric
in $\D$, making the connection with Ahlfors' derivative. Section 3 makes a summary
of Ahlfors' Schwarzian for curves and the injectivity criterion derived in \cite{chuaqui-gevirtz}. In Section 4 we state and prove
our main result, and draw various corollaries. The analysis based on Sturm comparison and 
the required regularity properties of the geodesics near $\partial\D$ are presented in
Section 5. Extremal lifts for the conditions are studied in Section 6 and the criterion involving geodesically convex minimal disk is established here. In the final section we describe the procedure
that yields the homeomorphic extension to 3-space.

\sm

     A planar harmonic mapping is a complex-valued harmonic function $f(z)$, $z=x+iy$,
defined on some domain $\Omega\subset\mathbb{C}$.  If $\Omega$ is simply
connected, the mapping has a canonical decomposition $f=h+\overline{g}$,
where $h$ and $g$ are analytic in $\Omega$ and $g(z_0)=0$ for some specified
point $z_0\in \Omega$.  The mapping $f$ is locally univalent if and only if
its Jacobian $|h'|^2 - |g'|^2$ does not vanish.  It is said to be
orientation-preserving if $|h'(z)|>|g'(z)|$ in $\Omega$, or equivalently if
$h'(z)\neq0$ and the dilatation $\omega=g'/h'$ has the property
$|\omega(z)|<1$ in $\Omega$.

     According to the Weierstrass--Enneper formulas, a harmonic mapping
$f=h+\overline{g}$ with $|h'(z)|+|g'(z)|\neq0$ lifts locally to map into
a minimal surface, $\Sigma$, described by conformal parameters if and only if
its dilatation $\omega=q^2$, the square of a meromorphic function $q$.
The Cartesian coordinates $(U,V,W)$ of the surface are then given by
$$
U(z)=\text{Re}\{f(z)\}\,,\quad
V(z)=\text{Im}\{f(z)\}\,,\quad
W(z)= 2\,\text{Im}\left\{\int_{z_0}^z h'(\zeta)q(\zeta)\,d\zeta\right\}\,.
%\tag2
$$
We use the notation
$$
\wf(z) = \bigl(U(z),V(z),W(z)\bigr)
$$
for the lifted mapping of $\Omega$ into $\Sigma$.  The height of the
surface can be expressed more symmetrically as
$$
W(z)= 2\,\text{Im}\left\{\int_{z_0}^z
\sqrt{h'(\zeta)g'(\zeta)}\,d\zeta\right\}\,,
$$
since a requirement equivalent to $\omega=q^2$ is that $h'g'$ be the
square of an analytic function.  The first fundamental form of the
surface is $ds^2=e^{2\sigma}|dz|^2$, where the conformal factor is
$$
e^\sigma = |h'|+|g'|\,.
$$
The Gauss curvature of the surface at a point $\wf(z)$ for
which $h'(z)\neq0$ is
\begin{equation}
K = - e^{-2\sigma}\Delta \sigma
= - \frac{4|q'|^2}{|h'|^2(1+|q|^2)^4}\,,  %\tag3
\label{eq:curvature}
\end{equation}
where $\Delta$ is the Laplacian operator.  Further information about
harmonic mappings and their relation to minimal surfaces can be found
in \cite{duren:harmonic}.

   For a harmonic mapping $f=h+\overline{g}$ with $|h'(z)|+|g'(z)|\neq0$,
whose dilatation is the square of a meromorphic function, we have
defined \cite{cdo:harmonic schwarzian} the {\em Schwarzian derivative} by the formula
\begin{equation}
\S f = 2\bigl(\sigma_{zz} - \sigma_z^2\bigr)\,, %\tag4
\label{eq:harmonic-schwarzian}
\end{equation}
where
$$
\sigma_z = \frac{\partial\sigma}{\partial z}
= \frac12 \left(\frac{\partial\sigma}{\partial x}
- i \frac{\partial\sigma}{\partial y}\right)\,, \qquad z = x+iy\,.
$$
Some background for this definition is discussed in Section 2.
With $h'(z)\neq0$ and $g'/h'=q^2$, a calculation ({\it cf}. \cite{cdo:harmonic schwarzian})
produces the expression
$$
\S f = \S h +\frac{2\overline{q}}{1+|q|^2}
\left(q'' - \frac{q'h''}{h'}\right) -4\left(\frac{q'\overline{q}}{1+|q|^2}
\right)^2\,.%\tag5
$$
As observed in \cite{cdo:harmonic schwarzian}, the formula remains valid if $\omega$ is
not a perfect square, provided that neither $h'$ nor $g'$ has a simple
zero.

     It must be emphasized that we are not requiring our harmonic
mappings to be locally univalent. In other words, the Jacobian need
not be of constant sign in the domain $\Omega$.  The orientation of
the mapping may reverse, corresponding to a folding in the associated
minimal surface.  It is also possible for the minimal surface to
exhibit several sheets above a point in the $(U,V)$--plane.  Thus the
lifted mapping $\wf$ may be univalent even when the underlying
mapping $f$ is not.

\section{Conformal Schwarzian and Ahlfors' Derivative}

In this section, we present the definition of the conformal Schwarzian that applies to immersions and to
conformal metrics. We will also present its relation to the Ahlfors derivative introduced in \cite{stowe:a-d}, when dealing with
Weierstrass-Enneper lifts of harmonic mappings.
The use of Ahlfors' Schwarzian for curves in the proof of Theorem 4.1 below avoids the need to
consider a Schwarzian derivative of harmonic mappings {\it relative} to conformal metrics in $\D$.
Nevertheless, the first term on the left-hand side in (\ref{injectivity}) below, indeed corresponds to the Schwarzian
of $f$ relative to the conformal metric in $\D$, making the connections with Corollary 12 in \cite{stowe:a-d}.

Furthermore, despite the apparent Euclidean nature of Ahlfors' Schwarzian, the chain rule and the natural parametrizations
of geodesics in the conformal geometry lead to the required lower bounds for the Hessian of the canonical function
{\it relative to the conformal metric} that are
of central use in Section 5 and 7.

\sm
The definition of conformal Schwarzian and its properties are suggested by the classical case, and have analogues there, but the generalization
must be framed in the terminology of differential geometry. We refer to \cite{os:sch} for the higher dimensional setting
and to \cite{co:extremal} for applications of convexity in 2 dimensions, similar to what we will do here
for harmonic mappings.

Let $\mathbf{g}$ be a Riemannian metric on the disk $\Bbb D$. We may assume that $\mathbf{g}$ is conformal to the Euclidean
metric, $\mathbf{g}_0=dx\otimes dx +dy \otimes dy= |dz|^2$. Let $\psi$ be a smooth function on $\Bbb D$ and form the symmetric 2-tensor
\begin{equation}
\Hess_{\mathbf{g}}(\psi) - d\psi \otimes d\psi.
\label{eq:Hessian}
\end{equation}
Here $\Hess$ denotes the Hessian operator. For example, if $\gamma(s)$ is an arc-length parametrized geodesic for $\mathbf{g}$, then
\[
\Hess_{\mathbf{g}}(\psi)(\gamma',\gamma') = \frac{d^2}{ds^2}(\psi\circ \gamma)\,.
\]
The Hessian depends on the metric, and since we will be changing metrics we indicate this dependence by the subscript $\mathbf{g}$.

With some imagination the tensor \eqref{eq:Hessian} begins to resemble a Schwarzian; among other occurrences in differential geometry,
it arises (in  2 dimensions) if one differentiates the equation that relates the geodesic curvatures of a curve for two conformal metrics.
Such a curvature formula is a classical interpretation of the Schwarzian derivative, see \cite{os:sch} and \cite{cdo:curvature}. The
trace of the tensor is the function
\[
\frac{1}{2}(\Delta_\mathbf{g} \psi - ||\grad_\mathbf{g}\psi||_\mathbf{g}^2),
\]
where again we have indicated by a subscript that the Laplacian, gradient and norm all depend on $\mathbf{g}$. It turns out to be most
convenient to work with a traceless tensor when generalizing the Schwarzian, so we subtract off this function times the metric
$\mathbf{g}$ and define the \textit{Schwarzian tensor} to be the symmetric, traceless, 2-tensor
\[
B_\mathbf{g}(\psi)={\Hess}_\mathbf{g}(\psi)-d\psi\otimes d\psi-\frac{1}{2}(\Delta_\mathbf{g} \psi-||\grad_\mathbf{g} \psi||^2)\mathbf{g}\,.
\label{eq:B-psi}
\]
Working in standard Cartesian coordinates one can represent
$B_{\mathbf{g}}(\psi)$ as a symmetric, traceless $2\times 2$ matrix, say of the form
\[
\begin{pmatrix}
a & -b \\
-b & -a
\end{pmatrix}\,.
\]
Further identifying such a matrix with the complex number $a+bi$ then allows us to associate the tensor $B_{\mathbf{g}}(\psi)$ with $a+bi$.

At each point $z\in\Bbb D$, the expression $B_\mathbf{g}(\psi)(z)$ is a bilinear form on the tangent space at $z$, and so its norm is
\[
||B_\mathbf{g}(\psi)(z)||_\mathbf{g} =\sup_{X,Y}B_\mathbf{g}(\psi)(z)(X,Y)\,,
\]
where the supremum is over unit vectors in the metric $\mathbf{g}$. If we compute the tensor with respect to the Euclidean metric
and make the identification with a complex number as above, then
\[
||B_\mathbf{g_0}(\psi)(z)||_\mathbf{g_0} = |a+bi|\,.
\]

Now, if $f$ is analytic and locally univalent in $\Bbb D$, then it is a conformal mapping of $\Bbb D$ with the metric $\mathbf{g}$
into $\Bbb C$ with the Euclidean metric. The pullback $f^*\mathbf{g}_0$ is a metric on ${\Bbb D}$ conformal to $\mathbf{g}$,
say $f^*\mathbf{g}_0 = e^{2\psi}\mathbf{g}$, and  the (conformal) Schwarzian of $f$ is now defined to be
\[
\S_\mathbf{g} f = B_\mathbf{g}(\psi)\,.
\]
If we take $\mathbf{g}$ to be the Euclidean metric then $\psi = \log|f'|$. Computing $B_{\mathbf{g}_0}(\log|f'|)$ and writing
it in matrix form as above results in
\[
B_{\mathbf{g}_0}(\log|f'|)=\left( \begin{array}{rr} \RR\,\S f & -{\rm Im}\,\S f \\ -{\rm Im}\,\S f & -\RR\,\S f
\end{array}\right) \, ,
\]
where $\S f$ is the classical Schwarzian derivative of $f$. In this way we identify $B_{\mathbf{g}_0}(\log|f'|)$ with  $\S f$.

Next, if $f=h+\overline{g}$ is a harmonic mapping of $\Bbb D$ and $\sigma = \log(|h'|+|g'|)$ is the conformal factor associated
with the lift $\wf$, we put
\[
\S f = \S_{\mathbf{g}_0} \wf = B_{\mathbf{g}_0}(\sigma).
\]
Calculating this out and making the identification of the generalized Schwarzian with a complex number produces
\[
B_{\mathbf{g}_0}(\sigma) = 2(\sigma_{zz} - \sigma_z^2)\,,
\]
which is the definition of $\S f$ given in \eqref{eq:harmonic-schwarzian}.

\medskip

In this context, the Ahlfors derivative $\mathcal{A}f$ relative to $\mathbf{g}_0$ introduced in \cite{stowe:a-d}  is related to the conformal Schwarzian by
the equation
$$\mathcal{A}f=\S f+\frac12|K\circ \wf|\mathbf{g}_0 \, .$$
The definition of $\mathcal{A}f$ gives a two-tensor for arbitrary conformal immersions, following partly the conformal Schwarzian, but it incorporates
information of the second fundamental form of the target when codimension exists. It gives back Ahlfors Schwarzian for curves (presented in the next section) when
the domain manifold is an interval, and vanishes for M\"obius transformations of $\mathbb{R}^n$. As pointed out by the author in \cite{stowe:a-d},
it is interesting that no such operator will exhibit in addition a general chain rule $\mathcal{A}(G\circ F)=\mathcal{A}F+F^*(\mathcal{A}G)$,
although the operator introduced will comply with this chain rule in many situations. We refer the reader to \cite{stowe:a-d} for the analysis leading
to the definition and for further details.

\section{Ahlfors' Schwarzian}

Ahlfors \cite{Ah1} introduced a notion of Schwarzian derivative
for mappings of a real interval into ${\Bbb R}^n$ by formulating
suitable analogues of the real and imaginary parts of $\S f$ for
analytic functions $f$.  A simple calculation shows that
$$
\text{Re}\{\S f\} = \frac{\text{Re}\{f'''\overline{f'}\}}{|f'|^2}
-3 \,\frac{\text{Re}\{f''\overline{f'}\}^2}{|f'|^4} + \frac32
\frac{|f''|^2}{|f'|^2}\,.
$$
For mappings $\varphi : (a,b) \rightarrow {\Bbb R}^n$ of class $C^3$ with
$\varphi'(x)\neq0$, Ahlfors defined the analogous expression
\begin{equation}
\label{eq:S1}
\S_1\varphi = \frac{ \varphi'''\cdot\varphi'}{|\varphi'|^2}
- 3\frac{(\varphi''\cdot\varphi')^2}{|\varphi'|^4}
+ \frac32\frac{|\varphi''|^2}{|\varphi'|^2}\,,
\end{equation}
where $\cdot$ denotes the Euclidean inner product
and now $|\mathbf{x}|^2=\mathbf{x}\cdot\mathbf{x}$ for $\mathbf{x}\in
{\Bbb R}^n$.  Ahlfors also defined a second expression analogous to
$\text{Im}\{\S f\}$, but this is not relevant to the present
discussion.

Ahlfors' Schwarzian is invariant under postcomposition with M\"obius
transformations; that is, under every composition of rotations,
magnifications, translations, and inversions in ${\Bbb R}^n$.  Only its
invariance under inversion
$$
\mathbf{x} \mapsto \frac{\mathbf{x}}{|\mathbf{x}|^2}\,, \qquad
\mathbf{x}\in{\Bbb R}^n\,,
$$
presents a difficulty; this can be checked by straightforward but tedious
calculation.  It should also be noted that $\S_1$ transforms as expected
under change of parameters.  If $x=x(t)$ is a smooth function with
$x'(t)\neq0$, and $\psi(t)=\varphi(x(t))$, then
\begin{equation}\label{reparam}
\S_1\psi(t) = \S_1\varphi(x(t))\,x'(t)^2 + \S x(t)\,.
\end{equation}

With the notation $v=|\varphi'|$,  and based on the Frenet--Serret formulas, it was shown in \cite{chuaqui-gevirtz}
that
\begin{equation}
\S_1\varphi = \left(\frac{v'}{v}\right)' - \frac12\left(\frac{v'}{v}\right)^2
+ \frac12 v^2 k^2 = \S(s) + \frac12 v^2 k^2\,,
\label{eq:S1-and-S}
\end{equation}
where $s=s(x)$ is the arc-length of the curve and $k$ is its curvature.  Our proof of Theorem 1 will be based on
the following result found in \cite{chuaqui-gevirtz}.
\begin{thma} Let $p(x)$ be a continuous function such that the
differential equation $u''(x)+p(x)u(x)=0$ admits no nontrivial solution
$u(x)$ with more than one zero in $(-1,1)$. Let $\varphi : (-1,1)\rightarrow
{\Bbb R}^n$ be a curve of class $C^3$ with tangent vector
$\varphi'(x)\neq 0$.  If $\S_1\varphi(x)\leq2p(x)$, then $\varphi$ is
univalent.
\end{thma}

If the function $p(x)$ of Theorem A is even,
then the solution $u_0$ of the differential equation $u''+pu=0$ with initial conditions $u_0(0)=1$
and $u_0'(0)=0$ is also even, and therefore $u_0(x)\neq0$ on $(-1,1)$,
since otherwise it would have at least two zeros.  Thus the function
\begin{equation}
\Phi(x) = \int_0^x u_0(t)^{-2}\,dt\,, \qquad -1 < x < 1\,,
\label{eq:Phi}
\end{equation}
is well defined and has the properties $\Phi(0)=0$, $\Phi'(0)=1$,
$\Phi''(0)=0$, $\Phi(-x)=-\Phi(x)$.  The standard method of reduction
of order produces the independent solution $u=u_0\Phi$ to $u''+pu=0$,
and so $\S \Phi=2p$. Note also that $\S_1\Phi=\S \Phi$, since
$\Phi$ is real-valued.  Thus $\S_1\Phi=2p$.

The next theorem, again to be found in \cite{chuaqui-gevirtz}, asserts that
the mapping $\Phi: (-1,1)\rightarrow {\Bbb R}\subset{\Bbb R}^n$ is extremal
for Theorem A if $\Phi(1)=\infty$, and that every extremal mapping
$\varphi$ is then a M\"obius postcomposition of $\Phi$.

\begin{thmb} Let $p(x)$ be an even function with the properties
assumed in Theorem A, and let $\Phi$ be defined as above.  Let
$\varphi: (-1,1)\rightarrow {\Bbb R}^n$ satisfy $\S_1\varphi(x)\leq2p(x)$
and have the normalization $\varphi(0)=0$, $|\varphi'(0)|=1$, and
$\varphi''(0)=0$.  Then $|\varphi'(x)|\leq\Phi'(|x|)$ for $x\in(-1,1)$,
and $\varphi$ has an extension to the closed interval $[-1,1]$ that is
continuous with respect to the spherical metric.  Furthermore, there are
two possibilities, as follows.

\sm
\noi $(i)$  If $\Phi(1)<\infty$, then $\varphi$ is univalent in $[-1,1]$
and $\varphi([-1,1])$ has finite length.

\sm
\noi $(ii)$ If $\Phi(1)=\infty$, then either $\varphi$ is univalent in
$[-1,1]$ or $\varphi=R\circ\Phi$ for some rotation $R$ of ${\Bbb R}^n$.
\end{thmb}

Note that in case $(ii)$ the mapping $\Phi$ sends both ends of the
interval to the point at infinity and is therefore not univalent in
$[-1,1]$. The role of $\Phi$ as an extremal for the harmonic univalence criterion
\eqref{injectivity} will emerge in the following sections. Two important corollaries
of Theorem B are that a curve $\varphi: (-1,1)\rightarrow {\Bbb R}^n$ satisfying $\S_1\varphi(x)\leq2p(x)$
for which $\varphi(1)=\varphi(-1)$ must take the closed interval $[-1,1]$ to a circle
or a line union the point at infinity, and that $\S_1\varphi\equiv 2p$.

\begin{rem}
A final important observation that will be used in Section 6, is that if
$\S_1\varphi< 2\pi^2/l^2$ on an interval $I$ of length $l$, then $\varphi$
is injective on the closed interval $\overline{I}$. To prove this, we may
assume that $I=(-l/2,l/2)$. Note that $p(x)\equiv \pi^2/l^2$
satisfies the hypothesis in Theorem A. The even solution $u_0$ above is
given by $\cos\left(cx\right)$, with $c=\pi/l$, and corresponding extremal $\Phi(x)=(1/c)\tan(cx)$
for which $\S \Phi=2\pi^2/l^2$.

The estimate $\S_1\varphi< 2\pi^2/l^2$ and Theorem A show that $\varphi$ is injective on $I$, and
Theorem B shows that the extension to $\overline{I}$ must remain injective, for otherwise $\S_1\varphi=2\pi^2/l^2$,
a contradiction.
\end{rem}

\section{Embedded minimal disks}

In this section we shall give a proof of the following criterion for the conformal parametrization of
minimal disks to be injective. This theorem corresponds to Corollary 12 in \cite{stowe:a-d} when the domain manifold $M=\D$
endowed with the conformal metric $\mathbf{g}$.
Follow up results having to do with continuous extension to the closed disk,
extremal configuration, and homeomorphic extension to all 3-space, will be studied in the final sections.

\begin{thm}\label{main result} Let $f$ be a harmonic mapping with dilatation $\omega=q^2$
the square of a meromorphic function in $\D$. Let $\mathbf{g}=e^{2\rho}\mathbf{g}_0$ be
a metric in $\D$ conformal to the Euclidean metric, and suppose that any two points
in $\D$ can be joined by a geodesic in the metric $\mathbf{g}$ of length less that
$\delta$, for some $0<\delta\leq\infty$. If
\begin{equation}\label{injectivity} \left|\S f -2\left(\rho_{zz}-\rho_z^2\right)\right|+e^{2\sigma}|K|
\leq \frac{2\pi^2e^{2\rho}}{\delta^2}+2\rho_{z\bar{z}}  \end{equation}
then the lift $\wf$ is injective in $\D$.
\end{thm}

The proof will be based on showing that under (4.1), the restriction of $\wf$ to
any geodesic is injective. To this end, we state without proof the following variant of Lemma 2  in \cite{cdo:holomorphic}.

\begin{lemma} \label{lemma:ahl-schwarzian}
Let $\wf :\D \rightarrow \Sigma$ be the lift of a harmonic mapping $f$
defined in $\D$. Let $\gamma(t)$ be a Euclidean arc-length
parametrized curve in $\D$ with curvature $\kappa(t)$, and
let $\varphi(t) = \wf(\gamma(t))$ be the corresponding
parametrization of $\Gamma = \wf(\gamma)$ on
$\Sigma.$ Let $V(t)$ be the Euclidean unit
tangent vector field along $\varphi(t)$, given by
\[
V(t) = \frac{\varphi'(t)}{|\varphi'(t)|} \,.
\]
If $I\!I$ stands for the second fundamental form on $\Sigma$ then

\begin{equation}
\S_1\varphi  = {\rm Re}\{\S f({\gamma})({\gamma}')^2\}
+\frac12e^{2\sigma({\gamma})}\left(|K(\varphi)|
+|I\!I(V,V)|^2\right)+
\frac12\kappa^2\,.
\end{equation}
\end{lemma}

With this, we can now prove Theorem 4.1.

\begin{proof}[Proof of Theorem \ref{main result}] Let $z_1, z_2\in\D$ be two points. By assumption,
there is a geodesic $\gamma$ in the conformal metric $e^{2\rho}\mathbf{g}_0$ of length $<\delta$ joining
the two points. Let $\gamma=\gamma(t)$ be a Euclidean arc-length parametrization, and let $t=t(s)$
be a change of parameters so that $s\rightarrow \gamma(t(s))$ is a unit-length parametrization
relative of the background metric. This means that
$$e^{\rho(\gamma(t(s)))}|\gamma'(t(s))|\frac{dt}{ds}=1 \, ,$$
or
$$\frac{dt}{ds}=e^{-\rho(\gamma(t(s)))} \, .$$
We shall compute Ahlfors' Schwarzian of the parametrization $\psi(s)=\wf(\gamma(t(s)))=\varphi(t(s))$, using the notation of
Lemma 4.2. The new parametrization is defined for $s$ on an interval $I$ of length less than $\delta$. Using (\ref{reparam}), we
have that
$$ \S_1\psi(s)=\S_1\varphi(t(s))(t'(s))^2+\S t(s) \, , $$
where the first term on the right hand side comes from Lemma 4.2. We compute now the second term:
$$ \S t(s)=\left(\frac{t''}{t'}\right)'-\frac12\left(\frac{t''}{t'}\right)^2 \, . $$
We introduce the Euclidean unit vectors $\hat{t}, \hat{n}$ given by $\hat{t}=\gamma'(t)$ and $\gamma''(t)=\kappa\hat{n}$.
Since $t'=e^{-\rho(\gamma)}$, and since $(dt/ds)=e^{-\rho}$, we see that
$$\frac{t''}{t'}=-e^{-\rho}\nabla\rho\cdot\hat{t} \, ,$$
and therefore
$$\left(\frac{t''}{t'}\right)'=-e^{-2\rho}\Hess(\rho)(\hat{t},\hat{t})-e^{-2\rho}(\nabla\rho\cdot\hat{n})\kappa
+e^{-2\rho}(\nabla\rho\cdot\hat{t})^2\, .$$ Because $\gamma$ is a geodesic in the conformal metric, we have that
$\kappa=\nabla\rho\cdot\hat{n}$, which gives
\begin{equation}\label{S_1 repram} e^{2\rho}\S t(s)=-\Hess(\rho)(\hat{t},\hat{t})+\frac12(\nabla\rho\cdot\hat{t})^2-\kappa^2 \, .
\end{equation}
Therefore
$$e^{2\rho}\S_1\psi=\RR\{\S f(\gamma)(\gamma')^2\}+\frac12e^{2\sigma}\left(|K|+|I\!I(V,V)|^2\right)+A \, ,$$
where
$$A=-\Hess(\rho)(\hat{t},\hat{t})+\frac12(\nabla\rho\cdot\hat{t})^2-\frac12(\nabla\rho\cdot\hat{n})^2 \, .$$
Using that
$$B_{\mathbf{g_0}}(\rho)(\hat{t},\hat{t})=\Hess(\rho)(\hat{t},\hat{t})-(\nabla\rho\cdot\hat{t})^2-
\frac12\left(\Delta\rho-|\nabla\rho|^2\right) \, ,$$
some algebraic manipulations give that
$$A=-B_{\mathbf{g_0}}(\rho)(\hat{t},\hat{t})-2\rho_{z\bar{z}} =-\RR\{2(\rho_{zz}-\rho_z^2)(\gamma')^2\}-2\rho_{z\bar{z}}\, .$$
Therefore
$$e^{2\rho}\S_1\psi=\RR\left\{\left(\S f(\gamma)-2(\rho_{zz}-\rho_z^2)\right)(\gamma')^2\right\}+
\frac12e^{2\sigma}\left(|K|+|I\!I(V,V)|^2\right)-2\rho_{z\bar{z}} $$
$$\leq \left|\S f(\gamma)-2(\rho_{zz}-\rho_z^2)\right|+e^{2\sigma}|K|-2\rho_{z\bar{z}}\, .$$

\noindent Finally, the inequality (\ref{injectivity}) implies that
\begin{equation}\label{S1 estimate} \S_1\psi\leq\frac{2\pi^2}{\delta^2} \, . \end{equation}
We appeal now to Theorem A with $p(x)=\pi^2/\delta^2$ on an open interval $J$ of length less than $\delta$ containing
the closed interval $\overline{I}$, to conclude that $\psi$ is injective.  This proves the theorem.

\end{proof}

We will draw some corollaries as particular important cases of this theorem. The first instance corresponds
to the case when the diameter $\delta=\infty$, that is, when the metric is complete.

\begin{cor}
Let $f$ be a harmonic mapping with dilatation $\omega=q^2$
the square of a meromorphic function in $\D$. Let $\mathbf{g}=e^{2\rho}\mathbf{g}_0$ be
a complete metric in $\D$ conformal to the Euclidean metric. If
\begin{equation}\label{complete} \left|\S f -2\left(\rho_{zz}-\rho_z^2\right)\right|+e^{2\sigma}|K|
\leq -\frac12\rho_{z\bar{z}}  \end{equation}
then the lift $\wf$ is injective in $\D$.
\end{cor}

A second set of cases arise when considering
$$e^{\sigma}=\frac{1}{(1-|z|^2)^t} \quad ,\; t\geq 0 \, .$$
The resulting conformal metrics have negative curvature and are complete for $t\geq 1$. For $0\leq t<1$ the disk
is still geodesically convex but has finite diameter given by
$$\delta=2\int_0^1\frac{dx}{(1-x^2)^t} \, ,$$
which can be expressed in terms of the Gamma function as
$$\delta=\sqrt{\pi}\frac{\Gamma(1-t)}{\Gamma(\frac32-t)} \, .$$

\begin{cor}
Let $f$ be a harmonic mapping with dilatation $\omega=q^2$
the square of a meromorphic function in $\D$. If either
\begin{equation}
\left|\S f-\frac{2t(1-t)\bar{z}^2}{(1-|z|^2)^2}\right|+e^{2\sigma}|K|
\leq \frac{2t}{(1-|z|^2)^2}\; \; ,\; t\geq 1
\end{equation}
or
\begin{equation}
\left|\S f-\frac{2t(1-t)\bar{z}^2}{(1-|z|^2)^2}\right|+e^{2\sigma}|K|
\leq \frac{2t}{(1-|z|^2)^2}+\frac{2\pi}{(1-|z|^2)^{2t}}
     \left(\frac{\Gamma (\frac{3}{2}-t)}{\Gamma (1-t)}\right)^2
 \; ,\;0\leq t<1
\end{equation}
then the lift $\wf$ is injective in $\D$.
\end{cor}

Three important instances of this corollary are obtained when setting $t=0, t=1$ and $t=2$,
which yields, respectively,
\begin{equation}\label{pi2} \left|\S f\right|+e^{2\sigma}|K|\leq\frac{\pi^2}{2} \, ,\end{equation}
\begin{equation}\label{nehari} \left|\S f\right|+e^{2\sigma}|K|\leq\frac{2}{(1-|z|^2)^2} \, ,\end{equation}
and
$$\left|\S f+\frac{4\bar{z}^2}{(1-|z|^2)^2}\right|+e^{2\sigma}|K|\leq \frac{4}{(1-|z|^2)^2}$$
as sufficient conditions for injectivity.

The third condition implies that
\begin{equation}\label{porky}\left|\S f\right|+e^{2\sigma}|K|\leq \frac{4}{1-|z|^2}\end{equation}
is sufficient for the injectivity of the lift. Indeed, if (\ref{porky}) holds, then
$$\left|\S f+\frac{4\bar{z}^2}{(1-|z|^2)^2}\right|+e^{2\sigma}|K|\leq \left|\S f\right|+\frac{4|z|^2}{(1-|z|^2)^2}+e^{2\sigma}|K|$$
$$\hspace{2,6in}\leq\frac{4}{1-|z|^2}+\frac{4|z|^2)}{(1-|z|^2)^2}= \frac{4}{(1-|z|^2)^2}\, .$$
Conditions (\ref{pi2}), (\ref{nehari}), and (\ref{porky}) were obtained in \cite{cdo:harmonic lift}.

\sm

The criteria (4.10) for $1\leq t\leq2$ and (4.11) without the
diameter term are generalizations of Ahlfors' condition for holomorphic $f$
\begin{equation}
\left|\S f-\frac{2c(1-c)\bar{z}^2}{(1-|z|^2)^2}\right|
\leq \frac{2|c|}{(1-|z|^2)^2}
\end{equation}
when $c$ is real. In (4.14) $c$ may be any complex number with $|c-1|<1$ \cite{Ah2}.
\me

We draw here two additional corollaries from our main result.

\begin{cor}
Let $f$ be a harmonic mapping with dilatation $\omega=q^2$
the square of a meromorphic function in $\D$. Let $\tau=\tau(z)$ be a real-valued
function in $\D$ satisfying
\begin{equation}|\tau_z|\leq \frac{c}{1-|z|^2} \quad , \; z\in\D \, ,\end{equation}
for some constant $c<1$. If
\begin{equation}
\left|\S f-2(\tau_{zz}-\tau_z^2)+\frac{4\bar{z}\tau_z}{1-|z|^2}\right|+e^{2\sigma}|K| \leq
\frac{2(1+(1-|z|^2)^2\tau_{z\bar{z}})}{(1-|z|^2)^2} \, ,
\label{eq:epstein-criterion}
\end{equation}
then the lift $\wf$ is injective in $\D$.
\end{cor}

This corollary constitutes a generalization of one of the main results in \cite{Ep1}. See also
\cite{AH} for even more general sufficient criteria for holomorphic mappings to be injective.

\begin{proof} Let $e^{\rho}=e^{\tau}/(1-|z|^2)$. The inequality (4.15)
guarantees that the radial derivative $\rho_r$ is positive for all $r$ sufficiently close
to 1. This implies that $\D$ is geodesically convex in the metric $e^{2\rho}\mathbf{g}_0$, and thus Theorem 4.2
is applicable. In this corollary we have excluded the diameter term appearing in (\ref{injectivity}).

\end{proof}

\begin{cor}
Let $f$ be a harmonic mapping with dilatation $\omega=q^2$
the square of a meromorphic function in $\D$. Suppose that
$$|\sigma_z|\leq \frac{c}{1-|z|^2} \quad , \; z\in\D \, ,$$
for some constant $c<1$. If
$$|z\sigma_z|+\frac14(1-|z|^2)e^{2\sigma}|K|\leq \frac{1}{1-|z|^2} $$
then the lift $\wf$ is injective in $\D$.
\end{cor}

Corollary 4.6 can be considered a generalization of the well known criterion for
univalence of Becker \cite{Be}.

\begin{proof} The condition on $\sigma_z$ ensures as before that $\D$ is geodesically convex with the metric  $e^{2\sigma}\mathbf{g}_0$.
The corollary follows at once by applying Theorem 4.2 with  $e^{\rho}=e^{\sigma}/(1-|z|^2)$.
\end{proof}

\section{Convexity and Continuous Extensions}

The purpose of this section is to establish the continuous extension to $\overline{\D}$ of
lifts satisfying (\ref{injectivity}). We begin with the following crucial lemma.

\begin{lemma} Let $f$ be a harmonic mapping with dilatation $\omega=q^2$
the square of a meromorphic function in $\D$. Let $\mathbf{g}=e^{2\rho}\mathbf{g}_0$ be
a metric in $\D$ conformal to the Euclidean metric, and suppose that (\ref{injectivity}) holds.
Then
$$u_f(z)=\sqrt{e^{\rho-\sigma}}$$ satisfies
$$\frac{d^2}{ds^2}u_f(\gamma(s))+\frac{\pi^2}{\delta^2}u_f(\gamma(s))\geq 0 \, ,$$
for any arc-length parametrized geodesic $\gamma(s)$ in the metric $\mathbf{g}$. In
particular, when $\mathbf{g}$ is complete then $u_f$ is convex in this metric.
\end{lemma}

\begin{proof} Let $\gamma=\gamma(s)$ be a geodesic in $\mathbf{g}$ parametrized by arc-length,
as was considered in the proof of Theorem 4.2, where it was shwon that the curve $\psi(s)=\wf (\gamma(s))$
satisfies (\ref{S1 estimate}). Let $\tau=\tau(s)$ be such that $\tau'(s)=|\psi'(s)|=e^{\sigma-\rho}$.
A well known fact that is easy to verify, ensures that the positive function
$U=(\tau')^{-1/2}=u_f(\gamma)$ satisfies
$$U''+\frac12(S\tau)U=0 \, .$$
Because of the definition of the $\S_1$ operator, we have that $S\tau\leq \S_1\psi\leq2\pi^2/\delta^2$, and thus
$$U''+\frac{\pi^2}{\delta^2}U\geq 0 \, ,$$
as desired.
\end{proof}

By means of comparison we will obtain from this lemma lower bounds for
the canonical function $u_f$ along geodesics, which will ensure  an extension of the lift along geodesics. A bit more
will be required to turn this information into a continuous extension to the closed disk $\overline{\D}$.
To this end, we introduce three conditions on the metric $\mathbf{g}$ that control the geometry of
the geodesics near $\partial\D$. The conditions are mild and far from restrictive, and were considered in \cite{co:extremal}
for the same purpose in the context of general criteria for injectivity for holomorphic functions defined
in $\D$.

Unless noted otherwise, in the remainder of this paper  we will {\it always assume} that the metric $\mathbf{g}$
has non-positive curvature, that is, that $\rho_{z\bar{z}}\geq 0$, and so we will not state this as a separate
assumption in any of our results.
Geometrically, the  main consequence of this is that geodesics cannot cross more than once in $\D$. We let
$l_{\mathbf{g}}$ denote the length function (of a curve) and $d_{\mathbf{g}}$ the distance (between points).

The first property has to do with extending geodesics to the boundary, and with reaching every boundary point
in this way.  We state the property  first as it often appears in the literature, but we must then say more to
distinguish the complete and the non-complete cases.

\begin{defn} The metric $\mathbf{g}$ on $\D$ has the {\em Unique Limit Point} property (ULP) if:

\sm
\noi(a) Let $z_0\in\D$. If $\gamma(t)$, $0\le t<T\le \infty$
is a maximally extended geodesic starting at $z_0$ then $\lim_{t\ra T} \gamma(t)$ exists (in the
Euclidean sense). We denote it by
$\gamma(T) \in \partial\D$.

\sm
\noi(b) The limit point is a continuous function of the initial direction at $z_0$.

\sm
\noi(c) Let $\zeta\in\partial\D$. Then there is a geodesic starting at $z_0$ whose limit point on $\partial\D$ is $\zeta$.

\end{defn}

We say a little more about part (c) in this condition. The assumption of non-positive
curvature implies that the limit point is a monotonic function
of the initial direction at the base point. Part $(b)$ requires that it is
continuous. It is conceivable that,  for some metrics,  all geodesics from
a base point might tend to the same limit point on the boundary, so the
mapping from initial directions to points on $\partial\D$ would reduce to
a constant. We want to avoid this degenerate situation and be certain that
every boundary point is `visible', so we include that fact  in the statement of (ULP).

(ULP) is a natural condition on complete metrics and is frequently formulated this way, if not
with this appellation.  For our work on boundary behavior in the non-complete case we have to
strengthen it slightly. Again take any base point
$z_0\in\D$ and consider geodesics from $z_0$ extended maximally to their unique limit points on
the boundary. In general, the length of such a geodesic as a function of the initial direction at
$z_0$ is lower semicontinuous, and for our arguments we need to know that it is continuous.   We let
(ULP*) mean (ULP)  {\it plus} the continuity of the length function. This is the assumption we
will often adopt in the non-complete case. In the complete  case the length function is the constant
function
$+\infty$  and the particular problems we encounter in the non-complete case do not come up; (ULP)
will suffice as is.

The conditions above must be hypotheses in many of our results, but none of them, alone or together,
is asking too much of a metric (see, for example, Theorems 7, 8 in \cite{co:extremal}.)

\begin{thm} Let $f$ be a harmonic mapping satisfying the hypotheses
in Theorem 4.1, and suppose that the metric $\mathbf{g}$ satisfies (ULP) if it is complete and (ULP*) if it is not
complete. Then $\wf$ admits a (spherically) continuous extension to
$\overline{\D}$.
\end{thm}

\begin{proof} Let $\Sigma=\wf(\D)$.  The proof is based entirely on the one given
for the corresponding theorem for holomorphic mappings in [7, Thm. 3]. We
include the proof for the convenience of the reader. We will show that small arcs on $S^1$, corresponding to
intervals of initial directions of geodesics from a base point, which parametrize small arcs on
$\partial\Sigma$. To obtain the requisite estimates we have to modify
$\wf$ by  M\"obius transformations of the range, and this is why the theorem is stated in terms of spherical continuity.
Composing $\wf$ with a M\"obius transformation will generally not preserve minimality, but as we have seen,
all proof are based on Ahfors' operator, which is preserved under such
compositions.

The proof is slightly different in the two cases
$\dl <\infty$ and
$\dl=\infty$. We consider first
$\dl<\infty$; thus (ULP*) is in force.  Let $\zeta_0\in\partial\D$ and let $\gamma_0$ be a geodesic
in
$\D$ ending at $\zeta_0$. Let
$z_0\in\gamma_0$ be a point of distance $<\dl/8$ from $\zeta_0$, and let $\theta_0$ be the direction of
$\gamma_0$
at $z_0$. Choose a small enough  neighborhood $V$ of initial directions about $\theta_0$ with
corresponding geodesics covering an  arc $I\subset\partial\D$ of limit points so that  the
distances between
$z_0$ and all such limit points is $\le \dl/4$.

Let $\theta\in V$ and let $\gamma(t)$, $0\le
t\le T_\theta$ be the corresponding geodesic starting at $z_0$ and ending at a point on $I\subset
\partial\D$.
Replace $\wf$ by $M\circ \wf$, where the M\"obius transformation $M$ is chosen so that the associatied function
$u_{M\circ \wf}$ satisfies
$$
\grad u_{M\circ \wf}(z_0)=0 \quad \mbox{and}\quad u_{M\circ f}(z_0)=1.
$$
We want to apply Lemma 5.1 to $u_{M \circ \wf}$ along the geodesics $\gamma$. Since $\S_1(M\circ \wf)=\S_1 \wf$,
we continue to write $\wf$ for $M\circ \wf$ and $u_{\wf}$ for $u_{M\circ \wf}$. The function $U(t)=u_{\wf}(\gamma(t))$ satisfies
$$
U'' \ge -\frac{\pi^2}{\dl^2}U, \quad U(0)=0,\quad U'(0)=1.
$$
From this,
$$
U(t)\ge \cos(\frac{\pi}{\dl}t),
$$
and so
$$
U(t)\ge \cos(\frac{\pi}{\dl}\frac{\dl}{4})=\frac{1}{\sqrt 2}.
$$
Note that since $u_{\wf}$ is  non-zero in the sector swept out by the geodesics $\gamma$, the mapping $\wf$ remains
away from infinity there.  Thus
$$
{|d\wf|}=e^{\sigma} \le 2e^\sg,
$$
 along $\gamma$, and
\begin{equation}
\int_\gamma {|d\wf|}\,|dz| \le 2l_g(\gamma)\le \frac{\dl}{2}. \label{eq:bounded-integral}
\end{equation}
This implies that
$$
\lim_{t\ra T_\theta} \wf(\gamma(t))
$$
exists. We denote the limit by $\wf(\gamma(T_\theta))$; it lies on $\partial\Sigma$.

We prove next that $\wf(\gamma(T_\theta))\in\partial\Sigma$ depends continuously on the initial
direction $\theta$ of the geodesic. Let $\gamma_1$, $0\le t\le T_{\theta_1}$ and $\gamma_2$, $0\le t\le
T_{\theta_2}$, be two geodesic rays starting at $z_0$ with $\theta_1,\theta_2\in V$. We need to
estimate the  distance between
$\wf(\gamma_1(T_{\theta_1}))$ and
$\wf(\gamma_2(T_{\theta_2}))$. Let
$0<\tau <\min\{T_{\theta_1},T_{\theta_2}\}$. Then
\begin{equation*}
\begin{split}
\label{eq:continuity} |\wf(\gamma_1(T_{\theta_1}))-\wf(\gamma_2(T_{\theta_2}))|&\le
 |\wf(\gamma_1(T_{\theta_1}))-\wf(\gamma_1(\tau))|+
|\wf(\gamma_1(\tau))-\wf(\gamma_2(\tau))|\\
&\qquad+|\wf(\gamma_2(T_{\theta_2}))-\wf(\gamma_2(\tau))|.
\end{split}
\end{equation*}
The terms $|\wf(\gamma_i(T_{\theta_i}))-\wf(\gamma_i(\tau))|$ are dominated by the tails of the
integrals in \eqref{eq:bounded-integral}
which are uniformly bounded by $\dl/2$. Now using the continuity of the length function in the
hypothesis (ULP*), there is a
$\tau_0$ so that both these terms are small for
$\tau_0\le \tau <\min\{T_{\theta_1},T_{\theta_2}\}$ if $|\theta_1-\theta_2|$ is small. The
remaining term can be controlled using the continuity of
$\wf$ and the fact that
$|\gamma_1(\tau)-\gamma_2(\tau)|$ is small if $|\theta_1-\theta_2|$ is small. These estimates prove that
the
 endpoints $f(\gamma(T_\theta))\in\partial\Sigma$,
$\gamma$ varying, depend continuously on the initial directions $\theta=\gamma'(0)$.

It remains to show that any point in $\partial\Sigma$ is the image $\wf(\gamma(T_\theta))$ as in the
construction above.  Let
$\omega\in\partial\Sigma$ and let
$\{w_n\}$ be a sequence of points in $\Sigma$ which converges to $\omega$. Choose a subsequence,
labeled the same way, of $z_n=\wf^{-1}(w_n)$ converging to a point $\zeta\in\partial\D$. Let
$z_0\in\D$ be a point of distance $<\dl/8$ from $\zeta$.

Let $g_1$ be the metric on $\Sigma$ obtained by pulling back the metric $g$ on $\D$ by $\wf^{-1}$. Thus
$\wf\colon(\D,g) \rightarrow (\Sigma,g_1)$ is an isometry. Let
let $\Gamma_n(t)$ be the $g_1$-geodesic
joining
$\wf(z_0)=w_0$ to $w_n$ with $\Gamma_n(0)=w_0$. Another subsequence, again labeled in the same way, of
the initial directions
$\Gamma_n'(0)$ converges to a direction which determines a geodesic $\Gamma$. Let
$\gamma=\wf^{-1}(\Gamma)$, $\gamma=\gamma(t)$, $\theta=\gamma'(0)$, $0\le t \le T_\theta$. Let
$\gamma_n=\wf^{-1}(\Gamma_n)$ and let
$t_n=l_g(\gamma_n)=l_{g_1}(\Gamma_n)$.
Write
\begin{equation*}
\begin{split}
|\wf(\gamma(T_\theta))-w_n|&=
|\wf(\gamma(T_\theta))-\wf(\gamma_n(t_n))|\\
&\le |\wf(\gamma(T_\theta))-\wf(\gamma(\tau))|+|\wf(\gamma(\tau))-\wf(\gamma_n(\tau))|\\
&\qquad +|\wf(\gamma_n(\tau))-\wf(\gamma_n(t_n))|.
\end{split}
\end{equation*}
As $\gamma_n'(0)\ra \gamma'(0)=\theta$, we conclude   for $n$
sufficiently large that $|\wf(\gamma(T_\theta))-w_n|$ can be made arbitrarily small by choosing $\tau$
close enough to $T_\theta$. Hence $\omega=\wf(\gamma(T_\theta))$. This completes the proof in the case
$\dl<\infty$.

\bigskip

We indicate now how the argument should be modified in the complete case
$\dl=\infty$. Choose a base point $z_0$, which is fixed for the entire argument. Let
$w_0=\wf(z_0)$.
The $g_1$-geodesic rays from $w_0$ can
be extended indefinitely, and we need to know that  they have a   limit.
Any such ray is the image under $\wf$ of a geodesic  $\gamma=\gamma(t)$,
$\gamma(0)=z_0$. Changing $\wf$ by an appropriate M\"obius transformation of the range, and maintaining the same
notation convention as above, we may
assume that
$U'(0)\ge c>0$. Then, as before
we have $U(t)\ge b+ct$, $t\ge 0$, and
\begin{equation}
\int_\gamma{|d\wf|}\,|dz| <\infty. \label{eq:bounded-integral-2}
\end{equation}
Thus $\lim_{t\ra \infty}\wf(\gamma(t))$ exists, and we denote
if by $\wf(\gamma(\infty))\in\partial\Sigma$.

For the continuity of   $\wf(\gamma(\infty))$ depending on
the initial directions at
$z_0$ we argue as follows. Take a geodesic $\gamma_1(t)$ from $z_0$. This time we modify $\wf$ by a
M\"obius transformation to change the gradient of $u_{\wf}$ at $z_0$ so that
$U'(0)\ge c>0$ for all rays from $z_0$ that form an angle of less than $\pi/4$ with $\gamma_1'(0)$.
This makes the
integrals in (\ref{eq:bounded-integral-2}) uniformly bounded over all such
rays, and $\wf$ uniformly bounded in the sector covered by the rays. From here the proof of
continuity, and that all of $\partial\Sigma$ is hit by the $\wf(\gamma(\infty))$, is almost identical to the
above. Only (ULP) is necessary.

\end{proof}

\section{Extremal Lifts}

\begin{defn} Let $f$ be a harmonic mapping satisfying the hypotheses of Theorem 4.1. We say that $\wf$ is an
{\em extremal lift} for (\ref{injectivity})
if the extension of $\wf$ to $\overline\D$ is not injective on $\partial\D$.  A geodesic $\gamma$ in $\D$ is called
an {\it extremal geodesic} if it joins two points on
$\partial\D$ where an extremal lift $\wf$ fails to be injective. The lift $\wf$ is called non-extremal if it remains injective
on $\overline{\D}$.

\end{defn}

\begin{defn} The metric $g$ on $\D$ has the {\em Boundary Points Joined} property (BPJ) if any two
points on $\partial\D$  can be joined by a  geodesic which
lies in $\D$  except for its endpoints.

\end{defn}

We now state

\begin{thm}
Let $\mathbf{g}$ have the properties (ULP) (or (ULP*)) and (BPJ). Then the following hold.

\sm
\noi (i) Equality holds in (\ref{injectivity}) for an extremal lift along an extremal geodesic.

\sm
\noi (ii) The image $\wf(\gamma)$ of an extremal geodesic under the extremal lift $\wf$ is a Euclidean circle
that is also a line curvature.

\sm
\noi (iii) The minimal surface $\Sigma=\wf(\D)$ is part of a catenoid.

\end{thm}

\begin{proof} Let $\wf$ be an extremal lift along an extremal geodesic $\gamma$ in $\D$.

\sm
\noi (i) Let $\psi(s)=\wf(\gamma(s))$ be the restriction of $\wf$ to $\gamma$, as considered in
the proof of Theorem 4.1. The parameter $s$ ranges over an interval $I$
of length $l=l_{\mathbf{g}}(\gamma)\leq\delta$. It was shown in the proof of the theorem that
$$ \S_1\psi\leq\frac{2\pi^2}{\delta^2} \, .$$
We claim that $l_{\mathbf{g}}(\gamma)=\delta$ and that
$$ \S_1\psi\equiv\frac{2\pi^2}{\delta^2} \, ,$$
$$|I\!I(V,V)|^2\equiv |K| \ .$$
To prove the claim, suppose that $l<\delta$.
Then
$$ \S_1\psi\leq\frac{2\pi^2}{\delta^2} < \frac{2\pi^2}{l^2}\, ,$$
which would imply by Remark 3.1 that $\psi$ cannot fail to be injective on $\overline{I}$, a contradiction.
Hence $l=\delta$. Theorem B now shows that $\S_1\psi\equiv\frac{2\pi^2}{\delta^2} \, .$ The proof
of Theorem 4.1 shows now that this is only possible if $|I\!I(V,V)|^2\equiv |K|$, and hence
equality must hold in (\ref{injectivity}) along $\gamma$.

\sm
\noi (ii) The equation $|I\!I(V,V)|^2\equiv |K|$, implies that $\wf(\gamma)$ must be a line of curvature, which
is also a circle by part (i).

\sm
\noi (iii) According to the Bj\"{o}rling problem (see, {\it e.g.}, \cite{minimal surfaces}, p. 121), there exists
a unique minimal surface
with a given real-analytic tangent plane along a given real-analytic arc. Because the normal vector to any planar line
of curvature
of a surface forms a constant angle with the normal to the surface (see, {\it e.g.}, \cite{docarmo}, p. 152),
there exists an appropriate circle
of revolution of a catenoid that is also a line of curvature, and which forms this same angle. By the above uniqueness result,
we conclude that $\wf(\gamma)$ must lie on a catenoid.

\end{proof}

We finish the paper with the following criterion.

\begin{thm} Let $\Sigma\subset\mathbb{R}^3$ be a
geodesically convex immersed minimal disk. Suppose that
\begin{equation}\label{g-convex} |K|\, \leq \, \frac{4\pi^2}{\delta^2} \, ,\end{equation}
where $K$ stands for the Gaussian curvature and $\delta$ for the diameter of $\Sigma$.
Then the following hold.

\sm \noi (i) The minimal disk $\Sigma$ is embedded.

\sm \noi (ii) The boundary $\partial\Sigma$ admits a continuous parametrization by the circle $\mathbb{S}^1$.

\sm \noi (iii) Suppose that any two points on $\partial\Sigma$ can be joined by a geodesic $\Gamma$ contained in $\Sigma$,
except for its endpoints on $\partial\Sigma$. If $\partial\Sigma$ is not a simple curve, then $\Sigma$ lies on a catenoid.
Equality holds in (\ref{g-convex}) along a geodesic $\Gamma$ that coincides with the unique geodesic on a catenoid that
is a circle of revolution.

\sm \noi (iv) The condition (\ref{g-convex}) is sharp.

\end{thm}

\begin{proof}

In the proof, we may assume that $\delta<\infty$, for otherwise $\Sigma$ reduces to a plane.

\sm \noi (i) Let $\wf$ be a conformal parametrization of $\Sigma$ defined on $\D$. We need to show
that $\wf$ is injective. Because $\Sigma$ is geodesically convex with diameter $\delta$, it follows that
pairs of points in $\D$ can be joined by a geodesic in the metric $e^{2\sigma}\mathbf{g}_0$ of length less than
$\delta$. A direct calculation shows that (\ref{g-convex}) corresponds to (\ref{injectivity}), and thus
$\wf$ is injective.

\sm \noi (ii) As a curve in space, the curvature $k$ of a geodesic $\Gamma\subset\Sigma$ is determined by the second
fundamental form $I\!I$, which is bounded by $\sqrt{|K|}$. Hence all such geodesics have uniformly bounded
curvature. Fix a base point $w_0\in\Sigma$, and consider geodesics with initial point $w_0$ as a function
of the angle $\theta\in\mathbb{S}^1$ on the tangent space at $w_0$ of the initial direction. Geodesics
cannot intersect and will leave any compact subset. Because the curvature $k$ is bounded, it is easy to see
that each such geodesic $\Gamma$ extended maximally will converge to a unique limit point $\zeta_{\theta}\in\partial\Sigma$.
Because two geodesics with sufficiently close initial data remain arbitrarily close on a given compact set,
and again because the remaining tails are controlled by the bound on its curvatures, it follows that the
limit point $\zeta_{\theta}$ depends continuously on $\theta$. It is also easy to see that any point on
$\partial\Sigma$ is of this form, proving part (ii) of the theorem.

\sm \noi (iii) Suppose that $\partial\Sigma$ is not simple. Hence there exist two values $\theta_1, \theta_2$
for which $\zeta_{\theta_1}=\zeta_{\theta_2}$. We may assume that the correspondence $\theta\rightarrow\zeta_{\theta}$
is one-to-one for $\theta_1<\theta<\theta_2$. By assumption, for small $\epsilon>0$, there exists a geodesic $\Gamma_{\epsilon}$ joining
the points $\zeta_{\theta_1+\epsilon}$ and $\zeta_{\theta_2-\epsilon}$, and as $\epsilon\rightarrow 0$, $\Gamma_{\epsilon}$ will
converge to a geodesic $\Gamma$ in $\Sigma$ closing up at $\zeta_{\theta_1}=\zeta_{\theta_2}$. Then
$l(\Gamma)\leq\delta$ and its curvature $k$ satisfies
$$|k|\leq\sqrt{|K|}\leq\frac{2\pi}{\delta} \, .$$
Consider an arc-length parametrization $\varphi:I\rightarrow\Gamma$, defined on an interval of length at most $\delta$.
Then
$$\S_1\varphi=\frac12k^2\leq\frac{2\pi^2}{\delta^2} \, .$$
As argued in the proof of part (i) of Theorem 6.2, we conclude that $l(\Gamma)=\delta$, $\S_1\varphi\equiv2\pi^2/\delta^2$,
and that $\Gamma$ is a circle. It is also a line of curvature because $I\!I$ must be maximal long $\Gamma$. This proves
part (iii).

\sm \noi (iv) The criterion is sharp in the following two senses. First of all, a configuration is possible
for which (\ref{g-convex}) holds with equality along a geodesic: consider $\Sigma$ to be a geodesic
ball of radius $\pi$ centered at a point $w_0$ on the geodesic of revolution $\Gamma$ of the catenoid obtained by rotation of the
curve $x=\cosh(z)$. The Gaussian curvature along $\Gamma$ is 1 in absolute value, so equality will hold in (\ref{g-convex})
because $\delta=2\pi$.
Everywhere else on $\Sigma$, $|K|<1$, so the criterion is satisfied.

On the other hand, the criterion is also sharp in the sense that the constant $4\pi^2/\delta^2$ cannot be improved; simply take
a geodesic ball as above of radius $r>\pi$ to violate an embedding.

\end{proof}

\section{Extensions to Space}

The purpose of this section is to derive an extension of certain lifts to
the entire 3-space that represents an analogue of the Ahlfors-Weill construction. The extension will be a consequence
of setting up appropriate circle bundles in domain and range that can be matched, for example,
by M\"obius transformations. By appealing to generalized best M\"obius approximations to the lift, a rather explicit
extension was obtained in \cite{AW} when $\mathbf{g}$ is the Poincar\'e metric, which under natural additional assumptions,
was shown to be quasiconformal. We will not pursue here similar considerations of quasiconformality. To establish the results, we
will assume that the conformal metric in $\mathbb{D}$ is {\it complete} and that it satisfies the conditions (ULP) and (BPJ). We will introduce
a variant of the notion of being non-extremal that will be satisfied, for example, whenever strict inequality holds in $(\ref{injectivity})$, and
which will be equivalent to being non-extremal when the metric is real analytic.

Throughout this section, $\wf$ will be assumed to satisfy (\ref{injectivity}). Lemma 5.1 ensures that $u_{\wf}$ is convex in the metric $\mathbf{g}$,
and because this property is based on estimating Ahlfors' Schwarzian, the functions $u_{M\circ\wf}$ will also be convex whenever $M$ is
a M\"obius transformation.
We will say that the {\it unique critical point property} (UCP) holds  if for every such shift, the function
$u_{M\circ\wf}$ exhibits at most one critical point in $\D$. If $\wf$ satisfies $(\ref{injectivity})$ with
a strict inequality everywhere, then $u_{\wf}$ and all $u_{M\circ\wf}$ will be strictly convex. Hence at most one critical point can occur
and the (UCP) condition will hold.

We first establish the connections between the (UCP) property and that of being non-extremal.
Recall that the lift $\wf$ admits a spherically continuous extension to the closed disk.

\begin{lemma} { Let $\zeta\in\partial\D$ and suppose that $\wf(\zeta)$ is a finite point. If $\gamma$ is a geodesic ray in $\D$ ending
at $\zeta$ then $\wf(\gamma)$ has finite length.}
\end{lemma}

\begin{proof} Let $z_0$ be the initial point of $\gamma$ and let $M$ be a M\"obius transformation such that $u_{M\circ\wf}$ has
positive derivative at $z_0$ in the direction of $\gamma$. It follows as  in (5.2) that $M\circ\wf(\gamma)$ has finite length. If $M$ fixes infinity,
that is, is affine, then$\wf(\gamma)$ will also have finite length. On the other hand, if $M$ is an inversion with some center $q\in\mathbb{R}^3$,
then $q\neq \wf(\zeta)$, for otherwise $M\circ\wf(\gamma)$ would have infinite length. We conclude that $\wf(\gamma)$ also
has finite length.

\end{proof}

\begin{lemma} {Suppose that $\mathbf{g}$ is real-analytic and that $\wf$ is non-extremal. Then (UCP) holds.}
\end{lemma}

\begin{proof} Suppose, by way of contradiction, that (UCP) does not hold. Then there exists $M$ such that $u_{M\circ\wf}$ has
two critical points, say $z_1, z_2\in\D$.  Because of convexity, $u_{M\circ\wf}$ attains its absolute minimum at $z_1, z_2$ and
also along the geodesic segment joining them.
Since the quantities involved are real analytic, we conclude that $u_{M\circ\wf}$ is constant along the entire geodesic $\gamma$
through $z_1, z_2$ extended up to the boundary in both directions to points $\zeta_1, \zeta_2\in\partial\D$. This implies that, up to
a constant factor, $|d(M\circ\wf)|=e^{\rho}$ along $\gamma$, and hence $M\circ\wf(\gamma)$ has infinite length in both directions. From Lemma 7.1
we see that $M\circ\wf(\zeta_1)=M\circ\wf(\zeta_2)$ must be the point at infinity. To conclude that
$\wf$ is extremal, we must show that the endpoints $\zeta_1, \zeta_2$ of the geodesic $\gamma$ are distinct. Suppose not, and let $D\subset\D$
be the region enclosed by $\gamma$ and its endpoint $\zeta_1=\zeta_2$.  The all geodesics starting from a fixed point $z_0\in \gamma$ pointing into $D$
must also converge to $\zeta_1$. If along any such geodesic the function $u_{M\circ\wf}$ became eventually increasing, then $M\circ\wf$ would
be finite at $\zeta_1$, a contradiction. Since $u_{M\circ\wf}(z_0)$ is already the minimum, then $u_{M\circ\wf}$ must be constant along all
such geodesics, or equivalently, $u_{M\circ\wf}$ is constant in $D$. The identity principle show that $u_{M\circ\wf}$ is constant in $\D$, and therefore
$M\circ\wf$ is constant and equal to infinity on $\partial\D$. The original lift $\wf$ is also constant on $\partial\D$. If this constant
is a finite point then the topological sphere $\wf(\overline{\D})$ would exhibit points of positive Gaussian curvature, which is impossible.
If the constant is the point at infinity, then the shift $M$ was not necessary to begin with, and $\wf$ satisfies (\ref{injectivity}) with $\rho=\sigma$.
The inequality forces the Gaussian curvature to be identically zero and we are back to the holomorphic case treated in \cite{co:extremal}.

\end{proof}

We now show the complementary implication.

\begin{lemma} {Suppose that  the (UCP) property holds. Then $\wf$ is not extremal.}
\end{lemma}

\begin{proof} Suppose that $\wf$ is extremal, and let $\gamma$ be a geodesic in $\D$ joining points $\zeta_1, \zeta_2\in\partial\D$ for which
$\wf(\zeta_1)=\wf(\zeta_2)$. Consider a M\"obius shift $M$ sending the common point to infinity. As we have seen before,  $u_{M\circ\wf}$ must be
constant along $\gamma$. To prove that the constant value, say $c$, is the minimum, it suffices to show that $u_{\wf}$ has a critical point on $\gamma$.  Let
$D_1, D_2$ be the two region into which $\gamma$ divides $\D$, and suppose that $u_{M\circ\wf}$ has no critical point along $\gamma$. Then the normal derivative
of $u_{M\circ\wf}$ along $\gamma$ must keep a constant sign, thus  $u_{M\circ\wf}$ must
be decreasing when moving away from $\gamma$ in the direction of, say, $D_1$. In other words, $u_{M\circ\wf}(z)\leq c$ for all $z\in D_1$, hence
$|d\wf(z)|\geq (1/c^2)e^{\rho}$ there. This implies that $M\circ\wf(\zeta)=\infty$ for all points on the arc $C=\partial\D\cap\partial D_1$.
Hence $\wf$ itself is constant on $C$. Since $\gamma$ is an extremal geodesic, we conclude that $\wf(D_1)$ is a minimal surface with boundary
a circle. This is readily seen to imply that the minimal surface must reduce to a plane, and we are back in the holomorphic case found in
\cite{co:extremal} where it is shown that, indeed, $c$ is the minimum value. This finishes the proof of the lemma.

\end{proof}

Let us now consider the following type of bundles of circles that fibre 3-space.
As a general configuration, let $B$ be a smooth, open surface in $\R$, and consider a family $\mathfrak{C}(B)$ of
Euclidean circles $C_p$ indexed by $p \in B$, at most one of which is a Euclidean line, having the properties:

\medskip
\noindent
(i) $C_p$ is orthogonal to $B$ at $p$ and $C_p \cap \overline{B}= \{p\}$;

\smallskip
\noindent(ii)  if $p_1 \ne p_2$ then $_{p_1}\cap C_{p_2} = \emptyset$;

\smallskip
\noindent(iii) $\bigcup_{p \in B} C_{p} = \mathbb{R}^3\setminus \partial B$.

\smallskip

We regard the point at $\infty$ as lying on the line in $\mathfrak{C}(B)$.
We refer to $p\in C_p$ as the base point. If $B$ is unbounded then there is no line in $\mathfrak{C}(B)$, for a line
would meet $\overline{B}$ at its base point and at the point at infinity, contrary to (i).

\medskip

The model case is $B=\D$, with $\mathfrak{C_0}=\mathfrak{C}(\D)$ being the collection of  circles $C_z$ orthogonal to the complex
plane passing through $z\in\D$ and its reflection $1/\bar{z}$. In this case, only the circle through the origin becomes a line.
In order to set up a bundle of this type when $B=\wf(\mathbb{D})$ will require (UCP) to hold. The bundle will be established
through a series of lemmas, and
it will be necessary to shift the lift $\wf$ by suitable
M\"obius transformations $M=M_q$ of the form
$$M(p)=\frac{p-q}{|p-q|^2} \, ,$$
for which the canonical function is given by
$$
u_{M\circ\wf}=|\wf-q|\,u_{\wf} \, .
$$

\begin{lemma} Let $\wf$ satisfy (\ref{injectivity}) and let $z_0\in\mathbb{D}$ be fixed.
Consider the set $C$ of points $q\in\mathbb{R}^3$
for which  $u_{M\circ\wf}$ has a critical point at $z_0$. Then

\smallskip
\noindent
(i) $C$ is a circle orthogonal to $\Sigma$ at $\wf(z_0)$ with radius $r(z_0)=
\displaystyle{\frac{e^{\sigma(z_0)}}{|\nabla \log u_{\wf}(z_0)|}}$;

\smallskip
\noindent(ii)  $C$ is symmetric with respect to the tangent plane to $\Sigma$ at $\wf(z_0)$;

\smallskip
\noindent(iii) $(C\setminus\{\wf(z_0)\})\cap\overline{\Sigma}=\emptyset$.

\end{lemma}

\begin{proof}
We show first the following basic fact. Let $D$ be a planar domain and let $e^{\tau}$ be a given positive function on $D$.
Then the
set of points $q\in\mathbb{R}^3$ for which $e^{\tau}|q-z|^2$ has a critical point at $z_0\in D$ is a circle orthogonal
to the plane passing through the
points $z_0$ and $z_0+1/\tau_z(z_0)$. Indeed, the critical point condition can be written in the form
$$\tau_{\bar{z}}(z_0)=\frac{p-z_0}{|q-z_0|^2} \, ,$$ where $p$ is the projection of $q$ onto the plane on which $D$ lies.
This last equation is readily seen to be equivalent to the condition
$$\left|q-z_0-\frac{1}{2\tau_z(z_0)}\right|^2=\left|\frac{1}{2\tau_z(z_0)}\right|^2 \, ,$$
which represents the claimed circle.

Since the statement in (i) involves only first order data, we may replace the minimal surface $\wf(\D)$ with the projection
onto the tangent plane at $\wf(z_0)$
of a neighborhood of $\wf(z_0)$ on the surface. This projection is planar domain $D$ as above, and the result follows after
considering $u_{M\circ\wf}$ as a function of the image point $w=\wf(z)$. This proves parts (i) and (ii).

To prove (iii) let $q\in C$, $q\neq \wf(z_0)$. Then
$u_{M\circ\wf}$ is convex and has a positive minimum point at $z_0$.
If $q$ were on $\Sigma$ or on its boundary, then $u_{M\circ\wf}$ would tend there to zero, a contradiction.

\end{proof}

For $z\in\D$ and $w=\wf(z)$,  the circle described in the lemma will be denoted by $C_w$. Through the following lemmas it will be shown that
the family of circles $\{C_w\}_{w\in\Sigma}$ constitutes a circle bundle of $\mathbb{R}^3$ with base $\Sigma$. This bundle will be denoted simply by $\mathfrak{C}$.
An important property is that, in the presence of a critical point in $\D$, the radius $r(z)$ of the circle $C_{\wf(z)}$ tends to zero as $z$ approaches $\partial\D$.
Indeed, we claim that for all points $z$ away from the unique critical point of $u_{\wf}$ we will have
$$ e^{-\rho}|\nabla u_{\wf}|\geq a>0 \, ,$$
for some absolute constant $a$. This estimate is direct consequence of the lower bound by linear growth $u_{\wf}$ along geodesics rays emanating from
the critical point. From this it follows that
\begin{equation}r(z)\leq \frac{1}{au_{\wf}} \, ,
\end{equation}
which tends to zero at the boundary.

\begin{lemma} Let $\wf$ satisfy (\ref{injectivity}). If the condition (UCP) holds then $\Sigma$ is bounded if and only if
$u_{\wf}$ has a critical point in $\D$.
\end{lemma}

\begin{proof}
Suppose first that $u_{\wf}$ has a critical point, say $z_0\in\D$. The (UCP) property implies that
$u_{\wf}$ must be strictly increasing along any geodesic ray $\gamma$ starting at $z_0$. Hence, as in (5.2), we conclude that $\wf$ remains bounded
on $\gamma$, with bounds depending on the constants $b,c$. Since these constant can be chosen independently on the direction of $\gamma$ at $z_0$,
we see that $\wf$ is bounded.

Suppose now that $\Sigma$ is bounded. Then Lemma 7.1 shows that $\wf(\gamma)$ has finite length for any geodesic ray $\gamma$ reaching $\partial\D$.
This implies that along any such ray starting, say, at the origin, $u_{\wf}$ must become eventually increasing. Therefore, $u_{\wf}$ attains an interior
minimum, proving the existence of the desired critical point.

\end{proof}

The lemma can equally well be stated for any M\"obius shift $M\circ\wf$, in particular, if $M(\Sigma)$ is bounded then
$u_{M\circ\wf}$ must have a critical point in $\D$. We can now show that $\mathfrak{C}$ is
a circle bundle of $\mathbb{R}^3$ with base $\Sigma$. Since the property (i) of such a bundle is met by Lemma 7.4,
we are to show the properties (ii) and (iii) of a circle bundle.

For (ii), let $C_{w_1}, C_{w_2}$ be two such circles having a common point $q$. If $q\in\Sigma$ then $q=w_1=w_2$, hence
the circles are the same. If $q\notin\Sigma$ then $u_{M_q\circ\wf}$ has a critical
point at $z_1=\wf^{-1}(w_1)$ and at $z_2=\wf^{-1}(w_2)$. The (UCP) property implies that $z_1=z_2$, hence
$C_{w_1}=C_{w_2}$.

For (iii), consider a point $q\notin\overline{\Sigma}$. Then  $M_q(q)=\infty\notin M_q(\overline{\Sigma})$, meaning that
$M_q(\Sigma)$ must be bounded. Lemma 7.5 implies that  $u_{M_q\circ\wf}$ has a critical point in $\D$ and therefore $q\in C_w$
for some $w=\wf(z)$.

We make two final observations before setting up the extension. First, it follows from part (i) of Lemma 7.4 that $C_w\in\mathfrak{C}$ becomes a line
exactly when $\wf^{-1}(w)$ is the unique critical point of $u_{\wf}$. Secondly, that the bundle $\mathfrak{C}$ with base $\Sigma=\wf(\D)$ is
{\it conformally natural}, in the following sense. For a fixed M\"obius mapping $M_0$, Lemma 7.4 produces a bundle $\mathfrak{D}$ over
the base $M_0(\Sigma)$, which is not difficult to see coincides with the collection of circles $M_0(C)$ for $C\in\mathfrak{C}$.
\smallskip

Let now $\wf$ be a lift satisfying (\ref{injectivity}) for which (UCP) holds. It is natural to consider a spatial extension
of the lift by matching the circles in the bundles $\mathfrak{C}_0$ and $\mathfrak{C}$
through the respective base points. For the actual pointwise correspondence between $C_z$ and $C_{\wf(z)}$ we use
an adequate affine mapping $D_z$ with $D_z(z)=\wf(z)$. With this, let $\E:\mathbb{R}^3\rightarrow\mathbb{R}^3$
be defined by
\begin{equation}
\E(p)=\left\{\begin{array}{lll} \wf(z) & , & p=z\in\overline{\D} \\ & &\hspace{.8in} \, . \\  D_z(p) & , & p\in C_{\wf(z)}    \end{array} \right .
\end{equation}
The mapping $\E$ is injective on $\overline{\D}$ because $\wf$ is non-extremal. The properties of circle bundles guarantees that $\E$
is injective elsewhere and also onto. It is also readily seen to be continuous at all points $p\notin\partial\D$, whereas the continutity
on $\partial\D$ is guaranteed by (7.1) when $\Sigma$ is bounded. Since the bundle $\mathfrak{C}$ transforms to the corresponding bundle
over the  base $M\circ\wf$ for any M\"obius shift, we conclude that the extension $\E$ is continuous in the spherical metric whether
$\Sigma$ is bounded or not.

\smallskip

As a final comment, we mention that $\E$, when restricted to points $p\in\mathbb{C}$ does give back the Ahlfors-Weill extension when
$\mathbf{g}$ is the Poincar\'e metric \cite{AW}.

\bibliographystyle{plain}
%\bibliography{refs}

\begin{thebibliography}{10}

\bibitem{Ah1} Ahlfors, L.V., {\it Cross-ratios and {Schwarzian} derivatives in {${\Bbb R}^n$}},
Complex Analysis: Articles dedicated to Albert Pfluger on the occasion of his 80th birthday,
Birkh{\"a}user Verlag, Basel, 1989, 1-15

\bibitem{Ah2} Ahlfors, L.V., {\it Sufficient conditions for quasi-conformal extension}.
Discontinuous groups and Riemann surfaces, Annals of Math. Studies 79 (1974), 23-29.

\bibitem{AH} Anderson, J.M. and Hinkkanen, A., {\it Univalence criteria and quasiconformal extensions},
Trans. Amer. Math. Soc. 324 (1991), 823-842

\bibitem{Be} Becker, J., {\it L\"ownersche Differentialgleichung und quasikonform fortsetzbare schlichte Funktionen}, J. Reine Angew. Math.
255 (1972), 23-43

\bibitem{Ch} Chuaqui, M., {\it A unified approach to univalence criteria in the unit disc}, Proc. Amer. Math. Soc. 123 (1995), 441-453

\bibitem{chuaqui-gevirtz} Chuaqui, M. and Gevirtz, J., {\it Simple curves in $\mathbb{R}^n$ and Ahlfors' Schwarzian derivative},
Proc. Amer. Math. Soc. 132 (2004), 223-230

\bibitem{co:extremal} Chuaqui, M. and Osgood, B., {\it General univalence criteria in the disk: extensions and
extremal funcions}, Ann. Acad. Scie. Fenn. Math. 23 (1998), 101-132

\bibitem{cdo:harmonic schwarzian} Chuaqui, M., Duren, P. and Osgood, B., {\it The Schwarzian derivative for harmonic mappings},
J. Anal. Math. 91 (2003), 329-351

\bibitem{cdo:curvature} Chuaqui, M., Duren, P. and Osgood, B., {\it Curvature properties of planar harmonic mappings}, Comput. Methods
and Function Theory 4 (2004), 127-142

\bibitem{cdo:harmonic lift} Chuaqui, M., Duren, P. and Osgood, B., {\it Univalence criteria for lifts of harmonic mappings to minimal surfaces},
J. Geom. Anal. 17 (2007), 49-74

\bibitem{cdo:holomorphic} Chuaqui, M., Duren, P. and Osgood, B., {\it Injectivity criteria for holomorphic curves in $\mathbb{C}^n$}, Pure Appl.
Math. Quarterly 7 (2011), 223-251

\bibitem{AW} Chuaqui, M., Duren, P. and Osgood, B., {\it Quasiconformal Extensions to Space of Weierstrass-Enneper Lifts}, submitted,
(arxiv:math.CV/1304.4198)

\bibitem{minimal surfaces} Dierkes, U., Hildebrandt, S., K\"{u}ster, A. and Wohlrab, O. {\it Minimal Surfaces I: Boundary Value Problems},
Springer-Verlag, 1992

\bibitem{docarmo} do Carmo, M., {\it Differential Geometry of Curves and Surfaces}, Prentice Hall, 1976

\bibitem{duren:harmonic}
Duren, P., {\it Univalent Functions}, Springer-Verlag, 1983.

\bibitem{Ep1} Epstein, Ch., {\it The hyperbolic Gauss map and quasiconformal reflections}, J. Reine Angew. Math. 380 (1987), 196-214

\bibitem{Ne} Nehari, Z., {\it The Schwarzian derivative and schlicht functions},
Bull. Amer. Math. Soc. 55 (1949), 545-551


\bibitem{os:sch} Osgood, B. and Stowe, D., {\it The Schwarzian derivative and conformal mapping of Riemannian manifolds}, Duke Math. J. 67 (1992), 57-97

\bibitem{stowe:a-d} Stowe, D., {\it An Ahlfors derivative for conformal immersions}, J. Geom. Anal. 25 (2015), 592-615.

\end{thebibliography}

\bi
\noi
{\small Facultad de Matem\'aticas, Pontificia Universidad Cat\'olica de Chile,
Casilla 306, Santiago 22, Chile,\, \email{mchuaqui@mat.puc.cl}

\end{document}